\documentclass[5p]{elsarticle}

\usepackage[usenames,dvipsnames,svgnames,table]{xcolor}
\usepackage{graphicx}

\usepackage{amsthm}
\usepackage{amsmath} 
\usepackage{amssymb}

\usepackage{enumitem}
\usepackage{hyperref}
\hypersetup{
    colorlinks=true,
    linkcolor=blue,
    filecolor=magenta,      
    urlcolor=blue,
		citecolor=red,
    pdftitle={Sharelatex Example},
    pdfpagemode=FullScreen,
}

\usepackage{tikz}
\usetikzlibrary{arrows}

\usepackage{natbib}        

\newtheorem{Satz}{Theorem} 
\newtheorem{Lemm}{Lemma} 
\newtheorem{Def}{Definition} 
\newtheorem{remark}{Remark}

\newtheorem  {proposition}[Satz] {Proposition}

\newtheorem{Ass}{Assumption}

\newcommand \R   {\mathbb{R}}

\newcommand \K   {\mathcal{K}}
\newcommand \Kinf{\mathcal{K_\infty}}
\newcommand \PD  {\mathcal{P}}
\newcommand \KL  {\mathcal{KL}}
\newcommand \LL  {\mathcal{L}}

\newcommand \Uc   {\mathcal{U}}
\newcommand \Dc   {\mathcal{D}}

\newcommand \Iff   {\Leftrightarrow}

\newcommand \eps {\varepsilon}

\newcounter{syscounter}
\newenvironment{sysnum}{\begin{list}{($\Sigma{\arabic{syscounter}}$)}%
{\settowidth{\labelwidth}{($\Sigma4$)}
\settowidth{\leftmargin}{($\Sigma4$)~}%
\usecounter{syscounter}}}
{\end{list}}

\journal{Systems \& Control Letters}

\begin{document}

\begin{frontmatter}

\title{Lyapunov characterization of input-to-state stability for
  semilinear control systems over Banach spaces} 

\author[Pas]{\href{http://www.mironchenko.com/index.php/en/}{Andrii Mironchenko}}
\ead{andrii.mironchenko@uni-passau.de}

\author[Pas]{\href{https://scholar.google.de/citations?user=GPXp0s0AAAAJ&hl=uk}{Fabian Wirth}}
\ead{fabian.(lastname)@uni-passau.de}

\address[Pas]{Faculty of Computer Science and Mathematics, University of Passau,
Innstra\ss e 33, 94032 Passau, Germany }

\begin{abstract}
We prove that input-to-state stability (ISS) of nonlinear systems over Banach spaces is equivalent to existence of a coercive Lipschitz continuous ISS Lyapunov function for this system. For linear infinite-dimensional systems, we show that ISS is equivalent to existence of a non-coercive ISS Lyapunov function and provide two simpler constructions of coercive and non-coercive ISS Lyapunov functions for input-to-state stable linear systems.
\end{abstract}


\end{frontmatter}




Input-to-state stability (ISS) was introduced by Sontag in his seminal paper \cite{Son89} and has since become a backbone of robust nonlinear control theory. Applications of ISS include robust stabilization of nonlinear systems \cite{FrK08}, design of nonlinear observers \cite{AnP09}, analysis of large-scale networks \cite{JTP94,DRW10} and numerous other branches of nonlinear control \cite{KoA01}.

The success of ISS theory of ordinary differential equations (ODEs) and the need
for  proper tools for robust stability analysis of partial differential
equations (PDEs) motivated the development of ISS theory in the infinite-dimensional
setting \cite{DaM13,MiI16,MaP11,JLR08,KaK16,KaJ11,Mir16}.

The two main lines of research within infinite-dimensional ISS theory 
are the development of a general ISS theory of evolution equations in Banach spaces
and the application of ISS to stability analysis and control of particular important PDEs.

The results in the first area include for instance small-gain theorems for
interconnected infinite-dimensional systems and their applications to
nonlinear interconnected parabolic PDEs over Sobolev spaces
\cite{DaM13,MiI15b} and characterizations of local and global ISS properties \cite{Mir16, MiW17b}. Within the second line of research, constructions of ISS Lyapunov functions for nonlinear parabolic systems over $L_p$-spaces \cite{MaP11}, for linear time-variant systems of conservation laws \cite{PrM12}, for nonlinear Kuramoto-Sivashinsky equation \cite{AVP16} have been investigated. Non-Lyapunov methods were successfully applied to linear parabolic systems with boundary disturbances in \cite{KaK16b}.

In this paper, we follow the first line of research and prove converse Lyapunov theorems for ISS of linear and semilinear evolution equations in Banach spaces.
For us the primary motivation comes from the papers \cite{LSW96, SoW95}, in which converse UGAS Lyapunov theorems
have been applied to prove, in the case of ODEs, the equivalence between
ISS and the existence of a smooth ISS
Lyapunov function. This result along with further restatements of ISS in
terms of other stability notions \cite{SoW95,SoW96} and small-gain
theorems \cite{JTP94, DRW10} is at the heart of ISS theory of systems of
ordinary differential equations.

In Section \ref{sec:input-state-stab} using the method from \cite{SoW95} and converse Lyapunov theorems for global asymptotic stability of systems with disturbances from 
\cite{KaJ11b} we prove that \textit{ISS is equivalent to the existence of a coercive, Lipschitz continuous ISS Lyapunov function}. Along the way, we show that ISS is equivalent to the existence of a globally stabilizing feedback which is robust to multiplicative actuator disturbances of bounded magnitude (weak uniform robust stability, WURS).

In Section~\ref{sec:LinSys} we provide simpler constructions of coercive and non-coercive ISS Lyapunov functions for linear infinite-dimensional systems with bounded input operators. In particular, we show that the existence of \textit{non-coercive ISS Lyapunov functions} is already sufficient for ISS of linear systems with bounded input operators.

Whether the existence of a non-coercive ISS Lyapunov function is sufficient for ISS of infinite-dimensional nonlinear systems is not completely clear right now, although some positive results based on non-Lyapunov characterizations of ISS property have been achieved in \cite{MiW17b}.
For systems without disturbances, it was shown in \cite{MiW16c} that non-coercive Lyapunov functions ensure uniform global asymptotic stability of the system, provided certain additional mild conditions hold.
Extension of these results to the systems with inputs is a challenging question for future research.

In Section~\ref{sec:conclusions} we conclude the results of the paper. Some of the results of this paper have been presented at 54th IEEE Conference on Decision and Control (CDC 2015) \cite{MiW15} and at 10th IFAC Symposium on Nonlinear Control Systems (NOLCOS 2016) \cite{MiW16b}.


Let $\R_+:=[0,\infty)$. For the formulation of stability properties the following classes of functions are useful:
\begin{equation*}
\begin{array}{ll}
 {\PD} &:= \left\{\gamma:\R_+\rightarrow\R_+\middle|\ \gamma\mbox{ is
       continuous}, \right. \\
 & \phantom{aaaaaaaaaaaaaaaaaaaaaa} \left. \gamma(r)=0 \Iff r=0  \right\}, \\
{\K} &:= \left\{\gamma \in \PD \left|\ \gamma \mbox{ is strictly increasing} \right. \right\}, \\
{\K_{\infty}}&:=\left\{\gamma\in\K\left|\ \gamma\mbox{ is unbounded}\right.\right\},\\
{\LL}&:=\left\{\gamma:\R_+\rightarrow\R_+\left|\ \gamma\mbox{ is continuous and strictly}\right.\right.\\
&\phantom{aaaaaaaaaaaaaa} \text{decreasing with } \lim\limits_{t\rightarrow\infty}\gamma(t)=0\},\\
{\KL} &:= \left\{\beta:\R_+\times\R_+\rightarrow\R_+\left|\ \beta \mbox{ is continuous,}\right.\right.,\\
&\phantom{aaaaaaa}\left.\beta(\cdot,t)\in{\K},\ \beta(r,\cdot)\in {\LL},\ \forall t\geq 0,\ \forall r >0\right\}. \\
\end{array}
\end{equation*}

For a normed space $X$, we denote the closed ball of radius
$r$ around $0$ by $\overline{B}_r$ or $\overline{B}_r^X$ if we want to
make the space clear.

Given normed space $X,W$, we call a function $f:X\to W$ locally Lipschitz
continuous, if for all $r>0$ there exists a constant $L_r$ such that 
\begin{equation*}
    \| f(x) - f(y) \|_W \leq L_r \|x -y\|_X \quad \forall x,y\in \overline{B}_r.
\end{equation*}
In the finite dimensional case, local Lipschitz continuity is sometimes
defined using neighborhoods of points, and in this case, this is of course
equivalent. Note that in the infinite-dimensional case it is necessary to
go to a definition on bounded balls as these are not compact. The
terminology we use here is consistent with \cite[p. 185]{Paz83}.  This
concept is called ``Lipschitz continuity on bounded balls'' in
\cite{CaH98}. 

\section{Input-to-state stability and weak uniform robust stability}
\label{sec:input-state-stab}

In this paper we consider infinite-dimensional systems of the form
\begin{equation}
\label{InfiniteDim}
\dot{x}(t)=Ax(t)+f(x(t),u(t)), \ x(t) \in X, u(t) \in U,
\end{equation}
where $A$ generates a strongly continuous semigroup of boun\-ded linear operators, $X$ is a Banach space and $U$ is a normed linear space of input values.
As the space of admissible inputs, we consider the space ${\cal U}$ of globally bounded, piecewise continuous functions from $\R_+$ to $U$.

%

In this paper we consider mild solutions of \eqref{InfiniteDim}, i.e. solutions of the integral equation
\begin{equation}
\label{InfiniteDim_Integral_Form}
x(t)=T_t x(0) + \int_0^t T_{t-s} f(x(s),u(s))ds 
\end{equation}
belonging to the class $C([0,\tau],X)$ for certain $\tau>0$. 
Here $\{T_t ,\ t \geq 0\}$ is the $C_0$-semigroup over $X$, generated by $A$.
For the notions from the theory of $C_0$-semigroups and its applications
to evolution equations we refer to \cite{CuZ95, CaH98}. In the sequel, we will write $\phi (t,x,u)$ to denote the solution corresponding to the initial condition $\phi(0,x,u)=x$ and the input $u\in \mathcal{U}$.

In the remainder of the paper we suppose that the nonlinearity $f$ satisfies the following assumption:
\begin{Ass}
\label{Assumption1}
Let $f:X \times U \to X$ be bi-Lipschitz continuous on bounded subsets, which means that two following properties hold:
\begin{enumerate}
    \item $\forall C>0 \; \exists L^1_f(C)>0$, such that $\forall x,y
          \in X$ with $\|x\|_X \leq C,\ \|y\|_X \leq C$ and $\forall v \in U$, it holds that
\begin{eqnarray}
\|f(x,v)-f(y,v)\|_X \leq L^1_f(C) \|x-y\|_X.
\label{eq:Lipschitz-1}
\end{eqnarray}
    \item $\forall C>0 \; \exists L^2_f(C)>0$, such that $\forall u,v
          \in U$ with $\|u\|_U \leq C,\ \|v\|_U \leq C$ and $\forall x \in X$, it holds that
\begin{eqnarray}
\|f(x,u)-f(x,v)\|_X \leq L^2_f(C) \|u-v\|_U.
\label{eq:Lipschitz-2}
\end{eqnarray}
\end{enumerate}
\end{Ass}

Due to standard arguments, Assumption~\ref{Assumption1} implies that mild solutions corresponding to any $x(0) \in X$ and any $u \in \Uc$ exist and are unique (actually, the second condition is too strong for mere existence and uniqueness, but we need it for the further development). 

We call the system forward complete, if for all initial conditions $x\in X$ and all
$u\in\mathcal{U}$ the solution exists on $\R_+$.

We treat $u$ as an external input, which may have significant influence on the dynamics of the
system. For the stability analysis of such systems a fundamental role is played by the concept of input-to-state stability, which unifies external and internal stability concepts.
\begin{Def}
\label{Def:ISS}
System \eqref{InfiniteDim} is called {\it input-to-state stable
(ISS)}, if it is forward complete and there exist $\beta \in \KL$ and $\gamma \in \K$ 
such that $\forall x \in X$, $\forall u\in \Uc$ and $\forall t\geq 0$ the following inequality holds
\begin {equation}
\label{iss_sum}
\| \phi(t,x,u) \|_{X} \leq \beta(\| x \|_{X},t) + \gamma( \|u\|_{\Uc}).
\end{equation}
\end{Def}

A key tool to study ISS is an ISS Lyapunov function. 
\begin{Def}
\label{def:ISS_LF}
A continuous function $V:X \to \R_+$  is called a non-coercive \textit{ISS
  Lyapunov function},  if $V(0) = 0$ and if there exist $\psi_2 \in \Kinf$, $\alpha \in \PD$ and $\chi \in \K$ so that
\begin{equation}
\label{LyapFunk_1Eig_noncoercive}
 0 < V(x) \leq \psi_2(\|x\|_X) \quad
\forall x \in X\setminus \{ 0 \}.
\end{equation}
and so that the Dini derivative of $V$ along the trajectories of the system
\eqref{InfiniteDim} satisfies the implication
\begin{equation}
\label{DissipationIneq}
\|x\|_X \geq \chi(\|u(0)\|_U) \quad \Rightarrow \quad \dot{V}_u(x) \leq -\alpha(\|x\|_X)
\end{equation}
for all $x \in X$ and $u\in \Uc$, where 
\begin{equation}
\label{UGAS_wrt_D_LyapAbleitung}
\dot{V}_u(x)=\mathop{\overline{\lim}} \limits_{t \rightarrow +0} {\frac{1}{t}\big(V(\phi(t,x,u))-V(x)\big) }.
\end{equation}
If, in addition, there exists $\psi_1 \in \Kinf$ such that
\begin{equation}
\label{LyapFunk_1Eig_LISS}
\psi_1(\|x\|_X) \leq V(x) \leq \psi_2(\|x\|_X) \quad \forall x \in X,
\end{equation}
then $V$ is called a coercive ISS Lyapunov function.
\end{Def}
In Definition~\ref{def:ISS_LF} we defined ISS Lyapunov function in the so-called implication form. For another (dissipative) definition of ISS Lyapunov functions and for the relation between these definitions please consult \cite{MiI16}.
We have the following result, see \cite[Theorem 1]{DaM13}.
\begin{proposition}
\label{Direct_LT_ISS_maxType}
If there exists a coercive ISS Lyapunov function for \eqref{InfiniteDim}, then \eqref{InfiniteDim} is ISS.
\end{proposition}



We intend to show that
\textit{
\begin{center}
ISS of \eqref{InfiniteDim} implies existence of a coercive, locally Lipschitz continuous Lyapunov function for \eqref{InfiniteDim}.
\end{center}
}
On this way we follow the method developed in \cite{SoW95} for systems described by ODEs.
In order to formalize the robust stability property of
\eqref{InfiniteDim}, we consider the problem of global stabilization of
\eqref{InfiniteDim} by means of feedback laws which are subject to
multiplicative disturbances with a magnitude bounded by $1$. To this end
let $\varphi:X \to \R_+$ be locally Lipschitz continuous and consider inputs
\begin{eqnarray}
u(t):=d(t)\varphi(x(t)), \quad t\geq 0,
\label{eq:Multiplicative_Feedbacks}
\end{eqnarray}
where $d \in \Dc := \{d:\R_+ \to D, \text{ piecewise continuous}\}$, $D :=\{d \in U: \|d\|_U \leq 1\}$.

Applying this feedback law to \eqref{InfiniteDim} we obtain the system 
\begin{eqnarray}
\label{eq:Modified_InfDimSys_With_Disturbances}
\dot{x}(t)&=&   Ax(t)+f(x(t),d(t)\varphi(x(t))) \nonumber\\
          &=:&  Ax(t) + g(x(t),d(t)).
\end{eqnarray}
Let us denote the solution of 
\eqref{eq:Modified_InfDimSys_With_Disturbances} at time $t$, starting at
$x \in X$ and with disturbance $d \in \Dc$ by
$\phi_{\varphi}(t,x,d)$.
On its interval of existence, $\phi_{\varphi}(t,x,d)$ coincides with
the solution of \eqref{InfiniteDim} for the input $u(t)=d(t)\varphi(x(t))$.
\footnote{Forward completeness of \eqref{InfiniteDim} does not imply forward completeness of \eqref{eq:Modified_InfDimSys_With_Disturbances}. For  example, consider $\dot{x} = -x + u$, $u(t) = d\cdot x^2(t)$ for $d >0$.}

\subsection{Basic properties of the closed-loop system}

The next lemma shows that $g$ in \eqref{eq:Modified_InfDimSys_With_Disturbances} is Lipschitz continuous.
\begin{Lemm}
\label{lem:Regularity_of_g}
Let $f$ be locally bi-Lipschitz continuous. 
Then $g$ is 
Lipschitz continuous on bounded subsets of $X$, uniformly with respect to the second argument, i.e.
$\forall C>0 \; \exists L_g(C)>0$, such that $\forall x,y \in \overline{B}_C$ and  $\forall d \in D$, it holds that
\begin{eqnarray}
\|g(x,d)-g(y,d)\|_X \leq L_g(C) \|x-y\|_X.
\label{eq:Lipschitz}
\end{eqnarray}
\end{Lemm}

\begin{proof}
Pick an arbitrary $C>0$, any $x,y \in \overline{B}_C$, and any $d \in D$. It holds
\begin{multline*}
\|g(x,d)-g(y,d)\|_X = \|f(x,d\varphi(x)) - f(y,d\varphi(y))\|_X \\
\leq \|f(x,d\varphi(x)) - f(y,d\varphi(x))\|_X\\
 + \|f(y,d\varphi(x)) - f(y,d\varphi(y))\|_X.
\end{multline*}
Since $\varphi$ is Lipschitz continuous, it is bounded on $\overline{B}_C$
by a bound $R$. According to Assumption~\ref{Assumption1} and as
$\|d\|_{U}\leq 1$, we can upper bound the first summand by
$L_f^1(R)\|x-y\|_X$ and the second by $L_f^2(R) |\varphi(x)-\varphi(y)|$. The
claim now follows from the local Lipschitz continuity of $\varphi$.
\end{proof}

In particular, Lemma~\ref{lem:Regularity_of_g} shows that the system \eqref{eq:Modified_InfDimSys_With_Disturbances} is well-posed, i.e. its solution exists and is unique for any initial condition and any disturbance $d$.

\begin{remark}
Lipschitz continuous feedbacks do not necessarily lead to Lipschitz continuous $g$ if $f$ is not Lipschitz with respect to inputs. Consider e.g.
$\dot{x}(t) = (u(t))^{1/3}$ and $u(t) := x(t)$.
\end{remark}

\begin{Def}
\label{Def_RFC}
System \eqref{eq:Modified_InfDimSys_With_Disturbances} is called robustly forward complete (RFC) 
if for any $C>0$ and any $\tau>0$ it holds that 
\[
\sup_{\|x\|_X\leq C,\: d\in \Dc,\: t \in [0,\tau]}\|\phi_{\varphi}(t,x,d)\|_X < \infty.
\]
\end{Def}

\begin{Def}
\label{axiom:Lipschitz}
We say that the flow of \eqref{eq:Modified_InfDimSys_With_Disturbances} is Lipschitz continuous on compact intervals, if 
for any $\tau>0$ and any $R>0$ there exists $L>0$ so that for any $x,y \in
\overline{B}_R^X$, for all $t \in [0,\tau]$ and for all $d \in \Dc$ it holds that 
\begin{eqnarray}
\|\phi_{\varphi}(t,x,d) - \phi_{\varphi}(t,y,d) \|_X \leq L \|x-y\|_X.
\label{eq:Flow_is_Lipschitz}
\end{eqnarray}    
\end{Def}


We will need the following result, see \cite[Lemma 4.6]{MiW16c}, showing the regularity properties of the system \eqref{eq:Modified_InfDimSys_With_Disturbances}.
\begin{Lemm}
\label{lem:Regularity}
Assume that
\begin{enumerate}
    \item[(i)] \eqref{eq:Modified_InfDimSys_With_Disturbances} is robustly forward complete.
    \item[(ii)] $g$ is Lipschitz continuous on bounded subsets of $X$, uniformly w.r.t. the second argument.
\end{enumerate}
Then \eqref{eq:Modified_InfDimSys_With_Disturbances} has a flow which is Lipschitz continuous on compact intervals.
\end{Lemm}

\begin{Def}
System \eqref{eq:Modified_InfDimSys_With_Disturbances} is called uniformly globally asymptotically
stable (UGAS)
if there exists a $\beta \in \KL$ such that 
\begin{equation}
\label{UGAS_wrt_D_estimate}
 d\in \Dc,\; x \in X,\; t \geq 0 \ \Rightarrow \ \|\phi_{\varphi}(t,x,d)\|_X \leq \beta(\|x\|_X,t).
\end{equation}
\end{Def}

UGAS can be characterized with the help of  uniform global attractivity.
\begin{Def}
\label{def:UniformGlobalAttractivity}
System \eqref{eq:Modified_InfDimSys_With_Disturbances} is called uniformly globally attractive (UGATT), if for any $r,\eps >0$ there exists $\tau=\tau(r,\eps)$ so that for all $d \in \Dc$ it holds that 
\begin{equation}
\|x\|_X \leq r, \; t \geq \tau(r,\eps) \quad \Rightarrow \quad \|\phi_{\varphi}(t,x,d)\|_X \leq \eps.
\label{eq:UAG_with_zero_gain}
\end{equation}
\end{Def}

\begin{Def}
\label{def:UniformGlobalStability}
System \eqref{eq:Modified_InfDimSys_With_Disturbances} is called uniformly globally stable (UGS), if there exists $\sigma \in \Kinf$ so that
\begin{equation}
d\in \Dc,\; x \in X,\; t \geq 0 \ \Rightarrow \ \|\phi_{\varphi}(t,x,d)\|_X \leq \sigma(\|x\|_X).
\label{eq:GS}
\end{equation}
\end{Def}

The following characterization of UGAS follows easily from \cite[Theorem 2.2]{KaJ11b}.
\begin{proposition}
\label{prop:UGAS_Characterization}
System \eqref{eq:Modified_InfDimSys_With_Disturbances} is UGAS if and only if \eqref{eq:Modified_InfDimSys_With_Disturbances} is  UGATT and UGS.
\end{proposition}


Coercive Lyapunov functions corresponding to UGAS property are defined as follows:
\begin{Def}
\label{def:UGAS_LF_With_Disturbances}
A continuous function $V:X \to \R_+$ is called a \textit{Lyapunov function} for \eqref{eq:Modified_InfDimSys_With_Disturbances},  if there exist
$\psi_1,\psi_2 \in \Kinf$ and $\alpha \in \Kinf$
such that 
\begin{equation}
\label{LyapFunk_1Eig_UGAS}
\psi_1(\|x\|_X) \leq V(x) \leq \psi_2(\|x\|_X) \quad \forall x \in X
\end{equation}
holds 
and Dini derivative of $V$ along the trajectories of the system \eqref{eq:Modified_InfDimSys_With_Disturbances} satisfies 
\begin{equation}
\label{DissipationIneq_UGAS_With_Disturbances}
\dot{V}_d(x) \leq -\alpha(\|x\|_X)
\end{equation}
for all $x \in X$, and all $d \in \Dc$.
\end{Def}

The following converse Lyapunov theorem will be crucial for our developments \cite[Section 3.4]{KaJ11b}:
\begin{Satz}
\label{LipschitzConverseLyapunovTheorem-1}
Let \eqref{eq:Modified_InfDimSys_With_Disturbances} be UGAS and let its
flow be Lip\-schitz continuous on compact intervals, then
\eqref{InfiniteDim} admits a locally Lipschitz continuous Lyapunov
function.
\end{Satz}

We will need the following property, which formalizes the robustness of \eqref{InfiniteDim} with respect to the feedback \eqref{eq:Multiplicative_Feedbacks}.
\begin{Def}
\label{def:WURS}
System \eqref{InfiniteDim} is called weakly uniformly robustly asymptotically
stable (WURS), if there exist a locally Lipschitz $\varphi:X \to \R_+$
 and $\psi \in \Kinf$
such that 
$\varphi(x) \geq \psi(\|x\|_X)$ and \eqref{eq:Modified_InfDimSys_With_Disturbances}
is uniformly globally asymptotically stable with respect to $\Dc$.
\end{Def}

The next proposition shows how the WURS property of system \eqref{InfiniteDim} reflects the regularity of the solutions of \eqref{eq:Modified_InfDimSys_With_Disturbances}.

\begin{proposition}
\label{prop:Flow_of_g}
Consider a forward complete system \eqref{InfiniteDim}.
Assume that
\begin{enumerate}
    \item[(i)] $f$ is bi-Lipschitz on bounded subsets  of $X$;
    \item[(ii)] \eqref{InfiniteDim} is WURS.
\end{enumerate}
Then for any $\varphi$ satisfying the conditions of
Definition~\ref{def:WURS}, the closed-loop system
\eqref{eq:Modified_InfDimSys_With_Disturbances} has a flow, which is Lipschitz continuous on compact intervals.
\end{proposition}

\begin{proof}
Since \eqref{InfiniteDim} is WURS and $\varphi$ is a stabilizing feedback
as required in Definition~\ref{def:WURS}, system
\eqref{eq:Modified_InfDimSys_With_Disturbances} is forward complete and
UGAS. Let $\beta\in{\cal KL}$ be a bound as in \eqref{UGAS_wrt_D_estimate}. Then, for any $C>0$ and any $\tau>0$ 
\[
\sup_{\|x\|_X\leq C,\ d\in \Dc,\ t \in [0,\tau]}\|\phi_{\varphi}(t,x,d)\|_X \leq \beta(C,0) < \infty.
\]

Assumption (i) together with Lemma~\ref{lem:Regularity_of_g} imply that
$g$ is locally Lipschitz continuous
uniformly in the second argument. Thus, all assumptions of Lemma~\ref{lem:Regularity} are satisfied, and the claim follows. 
\end{proof}

\subsection{Main result}

The objective of this paper is to prove that for system \eqref{InfiniteDim} (at least with bi-Lipschitz nonlinearities) the notions depicted in Figure~\ref{ISS_CLT} are equivalent.

\begin{remark}
    The reader familiar with the results in \cite{SoW95} will notice that our assumptions on the dependence on $u$ are stronger than in the finite-dimensional case. For system \eqref{InfiniteDim} we need to ensure existence of solutions if a feedback is applied. 
    In the finite-dimensional case, it is sufficient to assume continuity by Peano's theorem. This guarantees existence but not uniqueness, but for the stability arguments, this is not a major drawback. For system \eqref{InfiniteDim}
    continuity is in general not sufficient for the existence of solutions \cite{God75, HaJ10}.
\end{remark}

\begin{figure}[tbh]
\centering
\begin{tikzpicture}[>=implies,thick]
\node (ISS) at (1,5) {(1) is ISS};
\node (WURS) at   (1,3) {(1) is WURS};
\node (WURSLF) at (6,3) {$\exists$ LF for (11)};
\node (ISSLF) at (6,5) {$\exists$ ISS-LF for (1)};

\node (Lem4) at (0,4) {\scriptsize Lemma 4};
\node (The3) at (3.5,2.7) {\scriptsize Theorem 3};
\node (Lem5) at (7,4) {\scriptsize Lemma 5};
\node (Prop1) at (3.4,5.3) {\scriptsize Proposition 1};

\draw[thick,double equal sign distance,->] (ISS) to (WURS);
\draw[thick,double equal sign distance,->] (WURS) to (WURSLF);
\draw[thick,double equal sign distance,->] (WURSLF) to (ISSLF);
\draw[thick,double equal sign distance,->] (ISSLF) to (ISS);

\end{tikzpicture}
\caption{ISS Converse Lyapunov Theorem}
\label{ISS_CLT}
\end{figure}
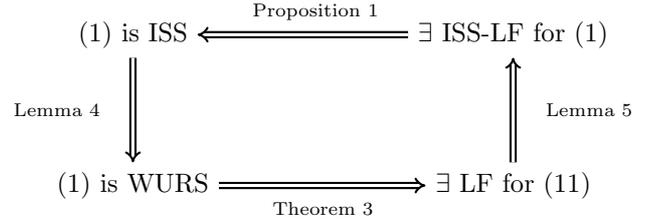
            
First, we show in Lemma~\ref{lem:ISS_implies_WURS} that ISS implies WURS. 
Next, we apply Theorem~\ref{LipschitzConverseLyapunovTheorem-1} to prove that WURS of \eqref{InfiniteDim} implies the existence of a Lipschitz continuous coercive ISS Lyapunov function for \eqref{InfiniteDim}.
Finally, the direct Lyapunov theorem
(Proposition~\ref{Direct_LT_ISS_maxType}) completes the proof.

\begin{Lemm}
\label{lem:ISS_implies_WURS}
If \eqref{InfiniteDim} is ISS, then \eqref{InfiniteDim} is WURS.
\end{Lemm}

\begin{proof}
The proof goes along the lines of \cite[Lemma 2.12]{SoW95}.

Let \eqref{InfiniteDim} be ISS. In order to prove that \eqref{InfiniteDim} is WURS we are going to use Proposition~\ref{prop:UGAS_Characterization}. 

Since \eqref{InfiniteDim} is ISS, there exist $\beta \in \KL$ and $\gamma \in \Kinf$ so that \eqref{iss_sum} holds for any $t \geq 0$, $x \in X$, $u \in \Uc$.
Define $\alpha(r):=\beta(r,0)$, for $r \in \R_+$. Substituting $u \equiv 0$ and $t=0$ into \eqref{iss_sum} we see that $\alpha(r) \geq r$ for all $r \in \R_+$.

Pick any $\sigma \in \Kinf$ so that $\sigma (r) \leq \gamma^{-1}\big(\frac{1}{4} \alpha^{-1}(\frac{2}{3}r)\big)$ for all $r\geq 0$.
We may choose locally Lipschitz continuous maps $\varphi:X\to \R_+$ and $\psi \in \Kinf$ such
that  $\psi(\|x\|_X) \leq \varphi(x) \leq \sigma(\|x\|_X)$
(just pick a locally Lipschitz continuous $\psi \in \Kinf$ and set
$\varphi(x):=\psi(\|x\|_X)$ for all $x \in X$, which guarantees that
$\varphi$ is locally Lipschitz continuous).

We are going to show that for all $x \in X$, all $t \geq 0$ and all $d \in \Dc$ it holds that
\begin{eqnarray}
\gamma\Big(\big\|d(t) \varphi(\phi_{\varphi}(t,x,d))\big\|_U\Big) \leq \frac{\|x\|_X}{2}. 
\label{eq:ISS_implies_WURS_Estimate}
\end{eqnarray}
First we show that \eqref{eq:ISS_implies_WURS_Estimate} holds for all
times $t\geq 0$ small enough. 
Since $\alpha^{-1}(r) \leq r$ for all $r >0$, we have
\begin{eqnarray*}
\hspace{-7mm}
\gamma\Big(\big\|d(t) \varphi(\phi_{\varphi}(t,x,d))\big\|_U\Big) &\leq& \gamma\big(\sigma(\|\phi_{\varphi}(t,x,d)\|_X)\big) \\
&\leq& \frac{1}{4} \alpha^{-1}\Big(\frac{2}{3}\|\phi_{\varphi}(t,x,d)\|_X\Big) \\
&\leq& \frac{1}{6} \|\phi_{\varphi}(t,x,d)\|_X.
\end{eqnarray*}
For any $d \in \Dc$ and any $x \in X$ the latter expression can be made smaller than 
$\frac{1}{2} \|x\|_X$ by choosing $t$ small enough, since $\phi_{\varphi}$
is continuous in $t$.

Now pick any $d \in \Dc$, $x \in X$ and define $t^* = t^*(x,d)$ by
\begin{equation*}
t^*:= \inf \left\{t \geq 0: \gamma\Big(\|d(t)\|_U \big|\varphi(\phi_{\varphi}(t,x,d))\big|\Big) > \frac{\|x\|_X}{2}\right\}.
\end{equation*}
By the first step we know $t^*>0$.
Assume that $t^*< \infty$ (otherwise our claim is true).
Then \eqref{eq:ISS_implies_WURS_Estimate} holds for all $t \in [0,t^*)$.
Thus, for all $t \in [0,t^*)$ it holds that 
\begin{eqnarray*}
\|\phi_{\varphi}(t,x,d)\|_X &\leq& \beta(\|x\|_X,t) + \frac{\|x\|_X}{2} \\
&\leq& \beta(\|x\|_X,0) + \frac{1}{2} \alpha(\|x\|_X) \\
& = & \frac{3}{2} \alpha(\|x\|_X).
\end{eqnarray*}
Using this estimate we find out that
\begin{align*}
\hspace{-10mm}
\gamma\Big(\|d(t^*)\|_U \big|\varphi(\phi_{\varphi}(t^*,x,d))\big|\Big) 
&\leq 
\frac{1}{4} \alpha^{-1}\Big(\frac{2}{3}\|\phi_{\varphi}(t^*,x,d)\|_X\Big)  \\
&\leq
\frac{1}{4} \alpha^{-1}(\|x\|_X)  \\
&\leq
\frac{1}{4} \|x\|_X.
\end{align*}
But this contradicts the definition of $t^*$.
Thus, $t^* = +\infty$.

Now we see that for any $x \in X$, any $d \in \Dc$ and all $t \geq 0$ we have 
\begin{equation}
\label{eq:Helps_for_RFC}
\|\phi_{\varphi}(t,x,d)\|_X \leq \beta(\|x\|_X,t) + \frac{\|x\|_X}{2},
\end{equation}
which shows uniform global stability of \eqref{eq:Modified_InfDimSys_With_Disturbances}.

\makebox[\linewidth][s]{Since $\beta \in \KL$, there exists a $t_1=t_1(\|x\|_X)$ so that } \\
$\beta(\|x\|_X,t_1) \leq \frac{\|x\|_X}{4}$ and consequently
\begin{equation*}
 d\in \Dc,\; x \in X,\; t \geq 0 \quad \Rightarrow \quad \|\phi_{\varphi}(t,x,d)\|_X \leq \frac{3}{4} \|x\|_X.
\end{equation*}

By induction we obtain that there exists a strictly increasing sequence of times $\{t_k\}_{k=1}^{\infty}$, which depends on the norm of $\|x\|_X$ but is independent of $x$ and $d$ so that
\begin{equation*}
\|\phi_{\varphi}(t,x,d)\|_X \leq \Big(\frac{3}{4} \Big)^{k} \|x\|_X,
\end{equation*}
for all $x \in X$, any $d \in \Dc$ and all $t \geq t_k$.

This means that for all $\eps>0$ and for all $\delta >0$ there exist a time $\tau=\tau(\delta)$ so that 
for all $x \in X$ with $\|x\|_X \leq \delta$, for all $d \in \Dc$ and for all $t \geq \tau$ we have
\[
\|\phi_{\varphi}(t,x,d)\|_X \leq \eps.
\]
This shows uniform global attractivity of \eqref{eq:Modified_InfDimSys_With_Disturbances}.

Now we are ready to apply Proposition~\ref{prop:UGAS_Characterization}, which shows that \eqref{eq:Modified_InfDimSys_With_Disturbances} is UGAS and thus 
\eqref{InfiniteDim} is WURS.
%
%
%
%
\end{proof}

%
%

%
%
%
%

\begin{Lemm}
\label{lem:WURS_implies_ISS_LF}
If \eqref{InfiniteDim} is WURS and Assumption~\ref{Assumption1}
is satisfied then there exists a locally Lipschitz continuous ISS Lyapunov
function for \eqref{InfiniteDim}. 
\end{Lemm}

\begin{proof}
Let \eqref{InfiniteDim} be WURS, which means that 
\eqref{eq:Modified_InfDimSys_With_Disturbances} is UGAS over $\Dc$ for
suitable $\varphi, \psi$ chosen in accordance with Definition~\ref{def:WURS}.
Proposition~\ref{prop:Flow_of_g} and Theorem~\ref{LipschitzConverseLyapunovTheorem-1} imply that
there exists a locally Lipschitz continuous
 Lyapunov function $V:X\to\R_+$, satisfying \eqref{LyapFunk_1Eig_UGAS}
for certain $\psi_1,\psi_2 \in \Kinf$ and whose Lie derivative along the solutions of \eqref{eq:Modified_InfDimSys_With_Disturbances} for all $x \in X$ and for all $d \in \Dc$ satisfies the estimate
\begin{eqnarray}
\dot{V}_d(x)\leq -\alpha(V(x)).
\label{eq:WURS_LF_Ineq}
\end{eqnarray}
This is equivalent to the fact that 
\begin{eqnarray}
\dot{V}_u(x)\leq -\alpha(V(x)).
\label{eq:WURS_LF_Ineq_2}
\end{eqnarray}
holds for all $x \in X$ and all $u \in \Uc$ satisfying $\|u\|_{\Uc} \leq \varphi(x)$.
This automatically implies that \eqref{eq:WURS_LF_Ineq_2} holds
for all $x \in X$ and all $u \in \Uc$ with $\|u\|_{\Uc} \leq \psi(\|x\|_X)$.

In other words, $V$ is an ISS Lyapunov function for \eqref{InfiniteDim} in an implication form with Lyapunov gain $\chi:=\psi^{-1}$.
\end{proof}

We conclude our investigation with the following characterization of ISS property:
\begin{Satz}
\label{ISS_Converse_Lyapunov_Theorem}
Let  Assumption~\ref{Assumption1} be fulfilled. Then the following statements are equivalent:
\begin{enumerate}
    \item \eqref{InfiniteDim} is ISS.
    \item \eqref{InfiniteDim} is WURS.
    \item There exists a  coercive ISS Lyapunov function for
          \eqref{InfiniteDim} which is locally Lipschitz continuous.
\end{enumerate}
\end{Satz}

\begin{proof}
The claim follows from Proposition~\ref{Direct_LT_ISS_maxType} and Lemmas~\ref{lem:ISS_implies_WURS} and \ref{lem:WURS_implies_ISS_LF}. 
\end{proof}

Theorem~\ref{ISS_Converse_Lyapunov_Theorem} shows that ISS is equivalent to the existence of a Lipschitz continuous \textit{coercive ISS Lyapunov function}. 
At the same time, the question whether the existence of a \textit{non-coercive ISS Lyapunov function} is sufficient for ISS of \eqref{InfiniteDim} remains open. This question is essentially infinite-dimensional, since in the ODE case non-coercive Lyapunov functions are automatically coercive, at least locally. In contrast to ODEs,
for linear infinite-dimensional systems, non-coercive ISS Lyapunov functions naturally arise when one constructs Lyapunov functions by solving Lyapunov operator equation, see \cite[Theorem 5.1.3 ]{CuZ95}. Hence it is of great interest to study criteria of ISS in terms of non-coercive ISS Lyapunov functions. In the next section, we show some preliminary results in this direction. An extensive treatment of this topic for nonlinear systems without inputs has been performed in \cite{MiW16c}.

\section{Linear systems}
\label{sec:LinSys}

In this section, we derive a converse Lyapunov theorem for linear systems
with a
bounded input operator $B$ of the form
\begin{equation}
\label{LinSys}
\dot{x}=Ax+Bu.
\end{equation}
The assumptions on $A$ are as before.
We start with a definition.
\begin{Def}
System \eqref{InfiniteDim} is {\it globally asymptotically
stable at zero uniformly with respect to the state} (0-UGAS), if there
exists a $ \beta \in \KL$, such that
\begin{equation}
\label{UniStabAbschaetzung}
x \in X,\ t \geq 0 \quad \Rightarrow \quad \left\| \phi(t,x,0) \right\|_{X} \leq  \beta(\left\| x \right\|_{X},t) .
\end{equation}
\end{Def}

Now we proceed with a technical lemma; its proof is straightforward and is omitted.
\begin{Lemm}
\label{AuxiliaryEquality}
Let $B \in L(U,X)$ and let $T$ be a $C_0$-semigroup. Then for any $u\in
{\cal U}$ it holds that
\begin{eqnarray}
\lim\limits_{h \rightarrow +0} \frac{1}{h} \int_0^h{T_{h-s} B u(s)ds} = B u(0).
\label{eq:Convolution_Bu0}
\end{eqnarray}
\end{Lemm}


The main technical result of this section is as follows:
\begin{proposition}
\label{ConverseLyapunovTheorem_LinearSystems}
%
If \eqref{LinSys} is 0-UGAS, then $V:X \to \R_+$, defined as
\begin{eqnarray}
V(x)=\int_0^{\infty} \|T_t x\|_X^2 dt
\label{eq:LF_LinSys_Banach}
\end{eqnarray}
is a non-coercive ISS Lyapunov function for \eqref{LinSys} which is
locally Lipschitz continuous.
 Moreover, $\forall x \in X$, $\forall u \in {\cal U}$ and $\forall \eps>0$ it holds that 
\begin{equation}
\label{Final_Lyapunov_Inequality}
\dot{V}_u(x) \leq -\|x\|_X^2 + \frac{\eps M^2}{2\lambda} \|x\|_X^2 + \frac{M^2}{2\lambda \eps}  \|B\|^2 \|u(0)\|_U^2,
\end{equation}
where $M, \lambda>0$ are so that 
\begin{equation}
\|T_t \| \leq M e^{-\lambda t}.
\label{eq:SemigroupEstimate}
\end{equation}
\end{proposition}

\begin{proof}
Let \eqref{LinSys} be 0-UGAS and pick $u\equiv 0$. 
Then \eqref{UniStabAbschaetzung} implies $\|T_t x\|_X \leq \beta(1,t)$ for
all $t \geq 0$ and for all $x$ with  $\|x\|_X=1$. Since $\beta \in \KL$,
there exists a $t^*$ such that  $\|T_{t^*}x\|_X < 1$ for all $x$,
$\|x\|_X=1$. Thus, $\|T_{t^*}\|<1$ and consequently $T$ is an exponentially
stable semigroup \cite[Theorem 2.1.6]{CuZ95}, i.e. there exist
$M,\lambda>0$ such that \eqref{eq:SemigroupEstimate} holds.

Consider $V:X \to \R_+$ as defined in \eqref{eq:LF_LinSys_Banach}.
We have
\begin{eqnarray}
V(x) \leq \int_0^{\infty}\hspace{-1.5mm} \|T_t \|^2\|x\|_X^2 dt \leq \frac{M^2}{2\lambda} \|x\|_X^2.
\label{eq:LF_Upper_estimate}
\end{eqnarray}

Let $V(x)=0$. Then $\|T_t x\|_X \equiv 0$ a.e. on $[0,\infty)$. Strong continuity of $T$ implies that $x = 0$, and thus 
\eqref{LyapFunk_1Eig_noncoercive} holds.

Next we estimate the Dini derivative of $V$: 
\begin{align*}
\dot{V}_u(x)& = \mathop{\overline{\lim}} \limits_{h \rightarrow +0} {\frac{1}{h}(V(\phi(h,x,u))-V(x)) } \\
    = & \mathop{\overline{\lim}} \limits_{h \rightarrow +0} {\frac{1}{h}\Big(\int_0^{\infty} \|T_t \phi(h,x,u)\|_X^2 dt - \int_0^{\infty} \|T_t x\|_X^2 dt \Big) } \\
    = & \mathop{\overline{\lim}} \limits_{h \rightarrow +0} \frac{1}{h}\Big(\int_0^{\infty} \Big\|T_t \Big(T_hx + \int_0^h{T_{h-s} B u(s)ds}\Big) \Big\|_X^2 dt \\
    & \quad -\int_0^{\infty} \|T_t x\|_X^2 dt \Big)  \\
    = & \mathop{\overline{\lim}} \limits_{h \rightarrow +0} \frac{1}{h}\Big(\int_0^{\infty} \Big\|T_{t+h} x + T_t \int_0^h{T_{h-s} B u(s)ds}\Big\|_X^2 dt \\
		    & \quad - \int_0^{\infty} \|T_t x\|_X^2 dt \Big)  
\end{align*}
\begin{align*}		
   \leq & \mathop{\overline{\lim}} \limits_{h \rightarrow +0} \frac{1}{h}\Big(\int_0^{\infty} \Big(\Big\|T_{t+h} x\Big\|_X \\
    &\quad + \Big\|T_t \int_0^h\hspace{-1.5mm}{T_{h-s} B u(s)ds}\Big\|_X \Big)^2 dt - \int_0^{\infty} \hspace{-1.5mm}\|T_t x\|_X^2 dt \Big)  \\
   = & I_1 + I_2,
\end{align*}
where
\begin{eqnarray*}
I_1:= \mathop{\overline{\lim}} \limits_{h \rightarrow +0} \frac{1}{h}\Big(\int_0^{\infty} \|T_{t+h} x\|_X^2 dt - \int_0^{\infty} \|T_t x\|_X^2 dt \Big)
\end{eqnarray*}
and
\begin{align*}
I_2  := &\mathop{\overline{\lim}} \limits_{h \rightarrow +0} \frac{1}{h}  \int_0^{\infty} \Big( 2\Big\|T_{t+h} x\Big\|_X \Big\|T_t \int_0^h\hspace{-1.5mm}{T_{h-s} B u(s)ds}\Big\|_X \\
 &\quad\quad\quad\quad + \Big\|T_t \int_0^h{T_{h-s} B u(s)ds}\Big\|_X^2 \Big) dt.
\end{align*}
Let us compute $I_1$:
\begin{eqnarray*}
I_1 &=& \mathop{\overline{\lim}} \limits_{h \rightarrow +0} \frac{1}{h}\Big(   \int_h^{\infty} \|T_t x\|_X^2 dt - \int_0^{\infty} \|T_t x\|_X^2 dt \Big) \\
    &=& \mathop{\overline{\lim}} \limits_{h \rightarrow +0} - \frac{1}{h}   \int_0^h \|T_t x\|_X^2 dt \\
    &=& -\|x\|_X^2.
\end{eqnarray*}
Now we proceed with $I_2$:
\begin{align*}
I_2 = & \mathop{\overline{\lim}} \limits_{h \rightarrow +0}  \int_0^{\infty}\hspace{-1.5mm} 2\Big\|T_{t+h} x\Big\|_X \Big\|T_t \frac{1}{h} \int_0^h\hspace{-1.5mm}{T_{h-s} B u(s)ds}\Big\|_X dt \\
      & \qquad\qquad+ \mathop{\overline{\lim}} \limits_{h \rightarrow +0} \int_0^{\infty}\hspace{-1mm} \frac{1}{h} \Big\|T_t \int_0^h{T_{h-s} B u(s)ds}\Big\|_X^2 dt.
\end{align*}
The limit of the second term equals zero since
\begin{multline*}
\mathop{\overline{\lim}} \limits_{h \rightarrow +0} \int_0^{\infty} \frac{1}{h} \Big\|T_t   \int_0^h{T_{h-s} B u(s)ds}\Big\|_X^2 dt \\
\leq \mathop{\overline{\lim}} \limits_{h \rightarrow +0} \int_0^{\infty} \frac{1}{h} M^4 e^{-2\lambda t}\|B\| \|u\|_{\cal U}
h^2 dt 
= \ 0.
\end{multline*}
To compute the limit of the first term, note that 
\begin{align*}
2\Big\|T_{t+h} x\Big\|_X \Big\|T_t \frac{1}{h} & \int_0^h{T_{h-s} B u(s)ds}\Big\|_X \\
\leq&
2 M \|x\|_X \|T_t \| M  \|B\| \sup_{r \in [0,h]}\|u(r)\|_{U} \\
\leq&
2 M^3 \|x\|_X  \|B\| \|u\|_{\cal U} e^{-\lambda t}.
\end{align*}
Thus, we can apply the dominated convergence theorem. Together with Lemma~\ref{AuxiliaryEquality} and Young's inequality this leads to
\begin{eqnarray*}
\hspace{-8mm}
I_2 & = & \int_0^{\infty} 2\|T_t x\|_X \|T_t  B u(0)\|_X dt \\
    & \leq & \int_0^{\infty} \eps \|T_t x\|_X^2 + \frac{1}{\eps} \|T_t  B u(0)\|_X^2 dt \\
    & \leq & \int_0^{\infty} \eps \|T_t \|^2 dt \|x\|_X^2 + \frac{1}{\eps} \int_0^{\infty} \|T_t \|^2 \|B u(0)\|_X^2 dt \\
    & \leq &  \frac{\eps M^2}{2\lambda} \|x\|_X^2 + \frac{M^2}{2\lambda \eps}  \|B\|^2 \|u(0)\|_U^2,
\end{eqnarray*}
for any $\eps >0$. 

Overall, we obtain that  $\forall x \in X$, $\forall u \in {\cal U}$ and for all $\eps>0$ the inequality \eqref{Final_Lyapunov_Inequality} holds. Considering $\eps < \frac{2\lambda}{M^2}$ this shows that $V$ is a non-coercive ISS Lyapunov function (in dissipative form) for \eqref{LinSys}.
It can be brought into implication form (as in \eqref{DissipationIneq}) by choosing the Lyapunov gain $\chi(s):=Rs$ for all $s \in \R_+$ and for $R$ large enough.

It remains to show the local Lipschitz continuity of $V$.
Pick arbitrary $r>0$ and any $x,y \in \overline{B}_r$. It holds that
\begin{align*}
|V&(x) - V(y)| = \Bigg| \int_0^{+\infty} \|T_t x\|^2_X - \|T_t y\|^2_X dt \Bigg| \\
&\leq   \int_0^{+\infty} \Big|\|T_t x\|_X - \|T_t y\|_X\Big| \big(\|T_t x\|_X + \|T_t y\|_X\big) dt \\
&\leq   \int_0^{+\infty} \|T_t x - T_t y\|_X \big(\|T_t x\|_X + \|T_t y\|_X\big) dt \\
&\leq   \int_0^{+\infty} Me^{-\lambda t} \|x - y\|_X Me^{-\lambda t} (\|x\|_X + \|y\|_X) dt \\
&\leq   \frac{M^2 r}{\lambda} \|x-y\|_X,
\end{align*}
which shows the Lipschitz continuity of $V$.
\end{proof}

\begin{remark}
The ISS Lyapunov function
$V$ defined in \eqref{eq:LF_LinSys_Banach} is not coercive in general. 
Noncoercivity of $V$ defined by \eqref{eq:LF_LinSys_Banach} implies that the system
\[\dot{x}=Ax, \quad y=x\]
is not exactly observable on $[0,+\infty)$ (even though we can measure the full state!), see \cite[Corollary 4.1.14]{CuZ95}.
The reason for this is that for any given exponential decay rate there are states that decay faster than this given rate, and thus we lose a part of the information about the state "infinitely fast".
\end{remark}

\begin{remark}
\label{rem:Non-coercive_ISS_LFs} 
Note that according to \cite[Section III.B]{MiW17b}, the existence of a non-coercive Lyapunov function satisfying 
\eqref{Final_Lyapunov_Inequality} ensures ISS of
\eqref{LinSys}.
\end{remark}

Below we provide another construction of ISS Lyapunov functions for the system \eqref{LinSys} with bounded input operators.
It is based on a standard construction in the analysis of $C_0$-semigroups, see e.g. \cite[Eq. (5.14)]{Paz83}.

For exponentially stable $C_0$-semigroup $T$ there exist $M,\lambda>0$ such that the estimate
\eqref{eq:SemigroupEstimate} holds. Choose $\gamma>0$ such that $\gamma -
\lambda < 0$. Then
\begin{equation}
    \label{eq:eqnorm}
    V^\gamma (x) := \max_{s\geq 0} \| e^{\gamma s} T_s x\|_X
\end{equation}
defines an equivalent norm on $X$, for which we have
\begin{multline}
    \label{eq:eqnorm2}
    V^\gamma(T_t x) = \max_{s\geq 0} \| e^{\gamma s} T_s T_t  x\|_X \\ =
    e^{-\gamma t} \max_{s\geq 0} \| e^{\gamma (s+t)} T_{s+t} x\|_X \leq
    e^{-\gamma t} V^\gamma(x) \,.
\end{multline}
Based on this inequality we obtain the following statement for ISS
Lyapunov functions.

\begin{proposition}
    \label{p:Lyapunovfunction2}
    Let \eqref{LinSys} be 0-UGAS. Let $M,\lambda>0$ be such that
\eqref{eq:SemigroupEstimate} holds and let $0<\gamma < \lambda$.
Then $V^\gamma :X \to \R_+$, defined by \eqref{eq:eqnorm}
is a coercive ISS Lyapunov function for \eqref{LinSys}. In particular, for
any $u\in {\cal U}$, $x\in X$, we
have the dissipation inequality
\begin{equation}
    \label{eq:dissipation2}
    \dot{V}^\gamma_u(x) \leq - \gamma \ V^\gamma(x) + V^\gamma (Bu(0)) \,.
\end{equation}
\end{proposition}

\begin{proof}
    In order to obtain the infinitesimal estimate, we compute, using the
    triangle inequality ($V^\gamma$ is a norm), the estimate
    \eqref{eq:eqnorm2}, and Lemma~\ref{AuxiliaryEquality},
    \begin{align*}
    \dot{V}^\gamma_u&(x) = \mathop{\overline{\lim}} \limits_{h \rightarrow +0} {\frac{1}{h}(V^\gamma(\phi(h,x,u))-V^\gamma(x)) } \\
    = & \mathop{\overline{\lim}} \limits_{h \rightarrow +0}
    \frac{1}{h}\Big( V^\gamma \Big(T_h x + \int_0^h{T_{h-s} B u(s)ds}\Big) - V^\gamma(x) \Big)  \\
    \leq & \mathop{\overline{\lim}} \limits_{h \rightarrow +0}
    \frac{1}{h}\Big( V^\gamma \big(T_h x\big) + V^\gamma
    \Big(\int_0^h{T_{h-s} B u(s)ds}\Big) - V^\gamma(x) \Big)\\
    \leq & \mathop{\overline{\lim}} \limits_{h \rightarrow +0}
    \frac{1}{h}\Big( (e^{-\gamma h} -1 ) V^\gamma(x) + V^\gamma
    \Big(\int_0^h{T_{h-s} B u(s)ds}\Big)  \Big)    \\
 \leq &  - \gamma \ V^\gamma(x) + V^\gamma (Bu(0)) \,.
    \end{align*}
This shows $V^\gamma$ is an ISS-Lyapunov function (in the dissipative form) and that
\eqref{eq:dissipation2} holds. Choosing a suitable Lyapunov gain $\chi \in\Kinf$, one can show that \eqref{DissipationIneq} holds and thus $V^{\gamma}$ is an ISS Lyapunov function in implication form. Coercivity is evident by construction.

It remains to show Lipschitz continuity of $V^\gamma$. Pick any $x,y \in X$ and assume that
$V^\gamma(x) > V^\gamma(y)$.
Then
\begin{align*}
V^\gamma(x) - V^\gamma(y)& = \max_{s\geq 0} \| e^{\gamma s} T_s x\|_X - \max_{s\geq 0} \| e^{\gamma s} T_s y\|_X \\
& \leq \max_{s\geq 0} \big(\| e^{\gamma s} T_s x\|_X - \| e^{\gamma s} T_s y\|_X\big) \\        
& \leq \max_{s\geq 0} \big| \| e^{\gamma s} T_s x\|_X - \| e^{\gamma s} T_s y\|_X \big| \\        
& \leq \max_{s\geq 0} \| e^{\gamma s} T_s(x-y)\|_X \\        
& \leq M\|x-y\|_X,        
\end{align*}
which shows that $V^\gamma$ is globally Lipschitz continuous. The case $V^\gamma(y) > V^\gamma(x)$ can be treated analogously.
\end{proof}


Finally, we can state the main result of this section:
\begin{Satz}
Let $B \in L(U,X)$. The following statements are equivalent:
\begin{enumerate}
    \item[(i)] \eqref{LinSys} is ISS.
    \item[(ii)] \eqref{LinSys} is 0-UGAS.
    \item[(iii)] $\{T_t\}_{t\geq 0}$ is an exponentially stable semigroup.
    \item[(iv)] $V$ defined in \eqref{eq:LF_LinSys_Banach} is a (not
          necessarily coercive) locally Lipschitz continuous 
ISS Lyapunov function for \eqref{LinSys}.
    \item[(v)] $V^\gamma$ defined in \eqref{eq:eqnorm} is a coercive globally Lipschitz continuous ISS Lyapunov function for \eqref{LinSys}.
\end{enumerate}
\end{Satz}

\begin{proof}
Equivalence between items (i) and (ii) can be easily derived from the variation of constants formula. The implications (ii) $\Rightarrow$ (iii) $\Rightarrow$ (iv) follow from Proposition~\ref{ConverseLyapunovTheorem_LinearSystems}.
Item (iv) implies (iii) due to Datko's Lemma, see \cite[Lemma 5.1.2, Theorem 5.1.3, p. 215]{CuZ95}. Implication (iii) $\Rightarrow$ (ii) is clear.
(ii) implies (v) due to Proposition~\ref{p:Lyapunovfunction2} and (v) implies (i) by Proposition~\ref{Direct_LT_ISS_maxType}.
\end{proof}

\section{Conclusions}
\label{sec:conclusions}

We have shown that input-to-state stability of a nonlinear infinite-dimensional system is equivalent to the existence of a coercive Lipschitz continuous ISS Lyapunov function. For linear systems, we have proposed simpler direct constructions of coercive as well as non-coercive Lipschitz continuous ISS Lyapunov functions.
Whether the existence of a non-coercive ISS Lyapunov function is sufficient for ISS of nonlinear infinite-dimensional systems, remains an open question.

\section*{Acknowledgements}

This research has been supported by the German Research Foundation (DFG) within the project
\href{http://www.fim.uni-passau.de/en/dynamical-systems/forschung/input-to-state-stability-and-stabilization-of-distributed-parameter-systems/}{"Input-to-}
\href{http://www.fim.uni-passau.de/en/dynamical-systems/forschung/input-to-state-stability-and-stabilization-of-distributed-parameter-systems/}{state stability and stabilization of distributed parameter}\\
\href{http://www.fim.uni-passau.de/en/dynamical-systems/forschung/input-to-state-stability-and-stabilization-of-distributed-parameter-systems/}{ systems"} (grant Wi 1458/13-1).
The authors thank Iasson Karafyllis for his helpful comments and suggestions.



\end{document}